\documentclass[12pt]{article}
\usepackage{amssymb}
\usepackage{amsmath}
\usepackage{amsthm}
\usepackage{authblk}
\usepackage{color}
\usepackage{geometry}		
\usepackage{graphicx}

\begin{document}

\newtheorem{theorem}{Theorem} 
\newtheorem{corollary}[theorem]{Corollary}
\newtheorem{lemma}[theorem]{Lemma}
\newtheorem{proposition}[theorem]{Proposition}

\title{Normal forms for saddle-node bifurcations: Takens' coefficient and applications in climate models}
\author[$\dagger$]{P.A.~Glendinning}
\author[$\ddagger$]{D.J.W.~Simpson}
\affil[$\dagger$]{Department of Mathematics, University of Manchester, Manchester, UK}
\affil[$\ddagger$]{School of Mathematical and Computational Sciences, Massey University, Palmerston North, New Zealand}

\maketitle


\begin{abstract}

We show that a one-dimensional differential equation depending on a parameter $\mu$ with a saddle-node bifurcation at $\mu =0$ can be modelled by an extended normal form $\dot y = \nu (\mu )-y^2+a(\mu )y^3$, where the functions $\nu$ and $a$ are solutions to equations that can be written down explicitly. The equivalence to the original equations is a local differentiable conjugacy on the basins of attraction and repulsion of stationary points in the parameter region for which these exist, and is a differentiable conjugacy on the whole local interval otherwise. (Recall that in standard approaches local equivalence is topological rather than differentiable.) The value $a(0)$ is Takens' coefficient from normal form theory.

The results explain the sense in which normal forms extend away from the bifurcation point and provide a new and more detailed characterisation of the saddle-node bifurcation. The one-dimensional system can be derived from higher dimensional equations using centre manifold theory. We illustrate this using two examples from climate science and show how the functions $\nu$ and $a$ can be determined analytically in some settings and numerically in others.

\end{abstract}

\section{Introduction}\label{section:intro}
The saddle-node bifurcation is
one of the two generic codimension-one
bifurcations of stationary points of differential equations
(the other being a Hopf bifurcation).
It occurs at parameters where the Jacobian matrix of a stationary point has a zero eigenvalue and two genericity conditions hold.
Near a saddle-node bifurcation the local
dynamics can be reduced to 
a one-dimensional differential equation on the associated extended centre manifold \cite{GH1983, Kuz1995}.

As parameters are varied to pass through the bifurcation,
a pair of stationary points is created or destroyed.
Restricted to the centre manifold, one stationary point is stable while the other is unstable.
On the centre manifold the system
is locally topologically equivalent (in ways that will be made more explicit in 
section~\ref{sect:basics}) to the truncated normal form
\begin{equation}\label{eq:tnf}
\dot{y} = \nu - y^2.
\end{equation}
This truncated normal form is often introduced via coordinate transformations which push the other terms in the Taylor expansion of the family of systems 
beyond quadratic.
These coordinate transformations can be given explicitly \cite[section 3.3]{Kuz1995}, yet the 
local equivalence is topological rather than differentiable except in the non-hyperbolic
case when the eigenvalue of the Jacobian matrix is zero. In this case successive coordinate changes can be made to remove
nonlinear terms of all orders except quadratic and cubic \cite{Takens1973}, see also section~\ref{sect:takens} below. This begs an obvious question. Why is
it possible to obtain stronger results (differentiable rather than topological) in the apparently
more exceptional non-hyperbolic case?

In this paper we show that the differentiable result
extends beyond the bifurcation value
provided a parameter-dependent cubic term is added to (\ref{eq:tnf}):
\begin{equation}\label{eq:nf}
\dot{y} = \nu(\mu) - y^2 + a(\mu) y^3.
\end{equation} 
For this modified equation the equivalence to the original equation restricted to the basin of attraction
of the stable stationary point and, separately, to the basin of repulsion of the unstable stationary point is differentiable provided the parameter-dependent coefficients $\nu$ and $a$ are chosen appropriately.

Here and below we assume a reduction to the centre manifold has been made,
so stability statements refer to stability restricted to the centre manifold.
To be more precise, suppose the restriction to the centre manifold is
\begin{equation}
\dot{x} = f(x,\mu),
\label{eq:f}
\end{equation}
and the saddle-node bifurcation occurs at $x=0$ when $\mu=0$.
Then $f(0,0) = f_x(0,0) = 0$
and for genericity we require $f_\mu(0,0) \ne 0$ and $f_{xx}(0,0) \ne 0$.
Without loss of generality, by changing of sign of the parameter and variable where necessary, we assume
\begin{equation}\label{eq:gen}
f(0,0)=0, \quad f_x(0,0)=0, \quad f_\mu (0,0)>0, \quad f_{xx}(0,0)<0.
\end{equation}
As shown in section~\ref{sect:normalforms}, to obtain a differentiable conjugacy between \eqref{eq:nf} and \eqref{eq:f}
it is necessary to have
\begin{equation}\label{eq:a0}
a(0)=\frac{2f_{xxx}}{3f_{xx}^2},
\end{equation}
evaluated at $x=\mu=0$.
This is exactly the quantity
obtained for the non-hyperbolic case $\mu=0$ by Takens \cite{Takens1973}, see section~\ref{sect:takens},
so we refer to \eqref{eq:a0} as \emph{Takens' coefficient} for the saddle-node bifurcation.
It provides a quantitative measure of how far the system is from the truncated normal form 
(\ref{eq:tnf}) and breaks the time-reversing symmetry $(y,t)\to (-y,-t)$ of \eqref{eq:tnf}.
It also reflects the degree of asymmetry in the two branches of stationary points in bifurcation diagrams,
and a large value of $a(0)$ indicates proximity to a cusp bifurcation.

The remainder of this paper is organised as follows.
In section~\ref{sect:takens} we revisit the formal power series approach
to differentiable equivalence at $\mu = 0$ and Takens' theorem for non-hyperbolic stationary points.
Then in section~\ref{sect:basics} we state and prove the main theoretical result on local equivalence.
Next we adapt the approach of \cite{GS2022} for bifurcations of maps to the continuous-time setting.
In section~\ref{sect:diffcon} we treat the $\mu < 0$ case that no stationary points exist locally,
then in section~\ref{sect:normalforms} treat the $\mu > 0$ case by the matching eigenvalues of the two stationary points to those of the
extended normal form \eqref{eq:nf} from which we see how the functions $\nu$ and $a$ are to be determined from $f$.

To illustrate the results we have chosen two applications in climate modelling.
In section~\ref{sect:Fraedrichs} we consider a simple temperature model of Fraedrich \cite{Fraedrich} and
show how Takens' coefficient can be expressed in terms of physical attributes of the system.
Then in section~\ref{sect:example2} we consider a version of Stommel's box model for ocean circulation \cite{Stommel1961}.
This model is two-dimensional but slow-fast which allows us to obtain a one-dimensional
approximation for which Takens' coefficient can be evaluated explicitly.
In section~\ref{sect:cm} we explain how the reduction to a centre manifold can be performed
and in section~\ref{sect:stommelcm} use this evaluate Takens' coefficient for Stommel's model numerically.
Finally section~\ref{sect:conclusion} provides concluding remarks.

\section{Coordinate transformations: power series}\label{sect:takens}

In standard normal form analysis the first step is to simplify \emph{when parameters are zero}
and then use a versal deformation argument to introduce general parameters to obtain
local behaviour \cite{GH1983,Wiggins2003}.
Coordinate changes are not performed away from the bifurcation point;
instead topological equivalence arguments are used, often without detailed justification.
On the other hand, to analyse saddle-node bifurcations using the implicit function theorem, the full strength of normal form theory is unnecessary and a shift of origin and scaling are used to remove the linear term in $x$ from the Taylor expansion of the differential equation and then it is pointed out that the resulting system is locally topologically equivalent to the truncated normal form (\ref{eq:tnf}), see for example \cite{Kuz1995}.

In this section we will recall the 
changes of coordinate at the bifurcation point for the saddle-node bifurcation
leading to Takens' normal form theorem (Theorem~\ref{thm:takens} below).

\subsection{Scaling the quadratic coefficient}

With $\mu = 0$ the general system \eqref{eq:f} with \eqref{eq:gen} has the form
\[
\dot{x} = \frac{1}{2} f_{xx} x^2 + \frac{1}{6} f_{xxx} x^3 + O(x^4),
\]
where here, and unless otherwise stated, derivatives are evaluated at $x = \mu = 0$.
In order to reach the modified normal form \eqref{eq:nf} with $\nu = 0$,
we first perform a linear change of coordinates, $y = \alpha x$.
We have
\begin{align*}
\dot{y} &= \alpha \left( \frac{1}{2} f_{xx} x^2 + \frac{1}{6} f_{xxx} x^3 \right) + O(x^4) \\
&= \frac{1}{2 \alpha} f_{xx} y^2 + \frac{1}{6 \alpha^2} f_{xxx} y^3 + O(y^4).
\end{align*}
Thus the choice $\alpha = -\frac{1}{2} f_{xx}$ yields
\begin{equation}
\dot{y} = -y^2 + a y^3 + O(y^4), \quad a =\frac{2f_{xxx}}{3f_{xx}^2}.
\label{eq:quadraticMatched}
\end{equation}
Already this simple calculation shows the origin of Takens' coefficient.
Notice our coordinate change $y = \alpha x$ is orientation-preserving
because $f_{xx} < 0$ by assumption.

\subsection{Persistence of the cubic term}

To the system \eqref{eq:quadraticMatched}
it is instructive to try and remove the cubic term via a subsequent coordinate change of the form
\[
z = y + \beta y^2.
\]
This inverts to $y = z - \beta z^2 + O(z^3)$, so
\begin{align*}
\dot{z} &= (1 + 2 \beta y) \dot{y} \\
&= -y^2 + (a - 2 \beta) y^3 + O(y^4) \\
&= -z^2 + a z^3 + O(z^4),
\end{align*}
and notice $\beta$ is absent from the cubic term.
Thus the cubic term cannot be removed by this change of coordinates.

\subsection{Removal of quartic and higher order terms}

However, it is possible to remove higher order terms. 
Write \eqref{eq:quadraticMatched} as
\[
\dot{y} = -y^2 + a y^3 + b y^k + O(y^{k+1}),
\]
where $k \ge 4$.
Then with
\[
z = y + \beta y^{k-1}, \quad
\text{so}~y = z - \beta z^{k-1} + O(z^k),
\]
we have
\begin{align*}
\dot{z} &= \left( 1 + \beta (k-1) y^{k-2} \right) \dot{y} \\
&= -y^2 + a y^3 + \left( b - \beta (k-1) \right) y^k + O(y^{k+1}) \\
&= -z^2 + a z^3 + \left( b - \beta (k-3) \right) z^k + O(z^{k+1}).
\end{align*}
Since $k>3$ we can choose
\[
\beta = \frac{b}{k-3}
\]
which removes the $z^k$ term.
This can be repeated for successively larger values of $k$ removing all
nonlinear terms of order four and above.

\subsection{Takens' theorem and general remarks}

It is one thing to show that there is a formal 
power series which removes all
terms higher than cubic,
it is quite another to show that this
formal power series converges on a neighbourhood of the stationary point. Takens resolved this issue
in the theorem quoted below. The aim of this paper is to determine the correct formulation which allows
us to accommodate $\mu \ne 0$.

\begin{theorem}(Takens \cite{Takens1973}) If $f$ is $C^\infty$ and satisfies \eqref{eq:gen}
then on a neighbourhood of zero the sequence of coordinate transformations defined
above converges to a $C^\infty$ change of 
variables in which the equation takes the normal form of
\begin{equation}\label{eq:Takenscoeff}
\dot{y} = -y^2 + a y^3, \quad a = \frac{2f_{xxx}}{3f_{xx}^2}.
\end{equation}
\label{thm:takens}\end{theorem}

We are not aware of explicit $C^k$ versions of this theorem, but there are corresponding discrete time formulations \cite{KCG1990},
see also \cite{GS2022}.

\subsection{Bifurcation Theorems}
In higher codimension problems, for example the Takens-Bogdanov bifurcation, and in some approaches to the 
Hopf bifurcation, a normal form is used at the bifurcation point and then a versal deformation argument is used to 
identify those small low-order terms (\emph{unfoldings}) in parameter and phase space
that imply all the topological behaviours close to the bifurcation point are realised \cite{GH1983, Kuz1995, Wiggins2003}.
For the simple codimension-one bifurcations this approach is not necessary because
topological equivalence is such a weak requirement.
Presumably it is for this reason that Theorem~\ref{thm:takens} is not often stated in the literature.  

\section{Equivalences and Conjugacy}\label{sect:basics}

Since formal proofs of even the local topological conjugacy results of bifurcation theorems are rarely given in textbooks, it is worth gathering together the definitions and conjugacy results needed in the remainder of this paper, and extending them where necessary.

\subsection{Smooth conjugacies}
Let $U, V \subset \mathbb{R}$ be open intervals
and $f: U \to \mathbb{R}$ and $g: V \to \mathbb{R}$ be $C^k$ with $k \ge 2$.
The differential equations
\[
\dot{x} = f(x), \quad \dot{y} = g(y),
\]
are said to be {\em $C^r$-conjugate} ($r \ge 1$)
if there exists a $C^r$ diffeomorphism $h : U \to V$ such that
\begin{equation}
g(h(x)) = h'(x)f(x), \quad \text{for all}~x \in U.
\label{eq:conjcondAlternate}
\end{equation}
This expresses the fact that $y = h(x)$ is a change of coordinates with
\[
\dot{y} = h'(x) \dot{x} = h'(x) f(x) = g(h(x)) = g(y).
\]
Now let $\phi_t(x)$ and $\psi_t(y)$ denote the flows induced by $\dot{x} = f(x)$ and $\dot{y} = g(y)$ respectively.
An equivalent formulation of \eqref{eq:conjcondAlternate} is
\begin{equation}\label{eq:conjcond}
h(\phi_t(x)) = \psi_t(h(x)).
\end{equation}

\subsection{Linearisation}

If $x^*$ is a stationary point of $f$ then $y^* = h(x^*)$ is a stationary point of $g$.
By differentiating \eqref{eq:conjcondAlternate} and setting $x = x^*$ we see that
the two stationary points have the same stability coefficient, i.e.
\[ f'(x^*) = g'(y^*).
\]

\begin{theorem}($C^k$-linearisation theorem)
Suppose $f : U \to \mathbb{R}$ is $C^k$ ($k \ge 2$)
and $x^* \in U$ is a stationary point of $\dot{x} = f(x)$
with $f'(x^*) = \lambda \ne 0$.
Then there exist neighbourhoods
$U_0 \subseteq U$ of $x^*$ and $V_0 \subset \mathbb{R}$ of $0$
such that $\dot x =f(x)$ on $U_0$ is $C^k$-conjugate to $\dot{y} = \lambda y$ on $V_0$.
\end{theorem}

This formulation is due to Sternberg \cite[Theorem 4]{Sternberg1957}.
In the general case of differential equations in $\mathbb{R}^n$ there are extra resonance conditions that need to hold between eigenvalues of the Jacobian matrices of $f$ and $g$ at the corresponding stationary points \cite{Belitskii1978, Sell1985,Sternberg1957}. In the one-dimensional case there is no resonance, and more generally some differentiability is possible even with resonance if $f$ is sufficiently smooth \cite{Guysinsky2003}.

Two differential equations satisfying the $C^k$-linearisation theorem for the same value of $\lambda \ne 0$
are both conjugate to $\dot{y} = \lambda y$, thus they are conjugate to each other.
That is, we have the following result.

\begin{corollary}
Suppose $f:U\to \mathbb{R}$ and $g:V\to \mathbb{R}$ are $C^k$ ($k \ge 2$) and
$\dot{x} = f(x)$ and $\dot{y} = g(y)$ have stationary points $x^* \in U$ and $y^* \in V$ with $f'(x^*) = g'(y^*) \ne 0$.
Then there exist neighbourhoods $U_0 \subseteq U$ of $x^*$ and $V_0 \subseteq V$ of $y^*$ such that
$\dot x=f(x)$ on $U_0$ and $\dot y =g(y)$ on $V_0$ are $C^k$-conjugate.
\label{cor:ckconj}\end{corollary}

\subsection{Extension to basins of attraction and repulsion}

We can now prove an adapted version of a theorem of Belitskii for discrete-time dynamical systems
that is key to the construction of the differentiable conjugacies for the bifurcations.
A similar result is given in \cite{GS2022} for the discrete-time setting.

\begin{theorem}
Suppose $f:U\to \mathbb{R}$ and $g:V\to \mathbb{R}$ are $C^k$ ($k \ge 2$) and
$\dot{x} = f(x)$ and $\dot{y} = g(y)$ have
precisely $n \ge 1$ stationary points,
$x_1 < x_2 < \cdots < x_n$ and $y_1 < y_2 < \cdots < y_n$ respectively,
with $f'(x_j) = g'(y_j) \ne 0$ for all $j \in \{ 1,2,\ldots,n \}$.
Write $U = (u_0,u_1)$ and $V = (v_0,v_1)$.
Then there exist $x_0 \in [u_0,x_1)$, $y_0 \in [v_0,y_1)$, $x_{n+1} \in (x_n,u_1]$, and $y_{n+1} \in (y_n,v_1]$ such that
$\dot{x} = f(x)$ on $(x_{j-1},x_{j+1})$ and 
$\dot{y} = g(y)$ on $(y_{j-1},y_{j+1})$ 
are $C^k$-conjugate for all $j \in \{ 1,2,\ldots,n \}$.
\label{thm:belit}\end{theorem}

\begin{proof}
As above let $\phi_t(x)$ and $\psi_t(y)$
denote the flows generated by $f$ and $g$ respectively.
Choose any $j \in \{ 1,2,\ldots,n \}$ and suppose $f'(x_j) > 0$
(the case $f'(x_j) < 0$ can be treated similarly).
By Corollary \ref{cor:ckconj} there exist open neighbourhoods $(u_-,u_+)$ of $x_j$ and $(v_-,v_+)$ of $y_j$
and a $C^k$-conjugacy $h:(u_-,u_+) \to (v_-,v_+)$ between $f$ and $g$, see Fig.~\ref{fig:hODEBelitskii}.

Now suppose $j \notin \{ 1,n \}$ (the extremal cases will be treated at the end).
Our task is to extend $h$ if necessary to the whole of
$(x_{j-1},x_{j+1})$ and $(y_{j-1},y_{j+1})$.
Since $f^\prime (x_j) > 0$, $f(x)>0$ for $x\in (x_j,x_{j+1})$ and so
$\phi_t(x) \to x_{j+1}$ and $\psi_t(y) \to y_{j+1}$ as $t \to \infty$ for all $x \in (x_j,x_{j+1})$ and $y \in (y_j,y_{j+1})$.
If $u_+ = x_{j+1}$ we are done, since clearly $\lim_{x \uparrow x_{j+1}} h(x) = y_{j+1}$.
So suppose $u_+ < x_{j+1}$ and let $v_+ = h(u_+)$ and notice $v_+ < y_{j+1}$.

\begin{figure}[h!]
\begin{center}
\includegraphics[height=6cm]{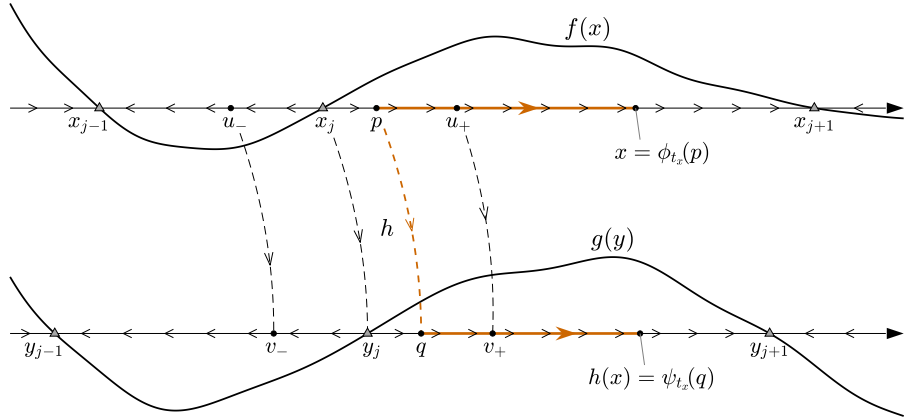}
\caption{
A sketch showing quantities introduced in the proof of Theorem \ref{thm:belit}. The path from $x$ backwards in time under $f$ to $p$, down to $q$ using $h$, and then forwards in time to $h(x)$ using $g$ shows how to extend the definition of $h$ via  \eqref{eq:defh}.
\label{fig:hODEBelitskii}
} 
\end{center}
\end{figure}

Let $p \in (x_j,u_+)$ and $q = h(p) \in (y_j,v_+)$.
For each $x \in (p,x_{j+1})$ there exists $t_x > 0$ such that $\phi_{t_x}(p) = x$
and clearly $t_x$ is an increasing function of $x$ with
$t_x \to \infty$ as $x \to x_{j+1}$.
Define $h(x)$ by using the two flows: from $x$ evolve back to $p$, transfer to $q = h(p)$, then evolve forward to $h(x)$, i.e.
\begin{equation}\label{eq:defh}
h(x) = \psi_{t_x}(q) = \psi_{t_x} \!\left( h \left( \phi_{-t_x}(x) \right) \right).
\end{equation}
This is illustrated in Fig.~\ref{fig:hODEBelitskii}.
Since the flows are $C^k$ (\cite{PdM1982}, Theorem 0.8), by construction $h(x)$ is $C^k$ and satisfies
the conjugacy condition (\ref{eq:conjcond}) for all $x \in (p,x_{j+1})$.
Thus we have extended $h$ to $(u_-,x_{j+1})$.
The same argument in $(x_{j-1},x_j)$ allows us to extend $h$ to $(x_{j-1},x_{j+1})$.
(If $f^\prime (x_j)<0$ then the argument is similar but with negative time.)
  
Finally we treat the case $j=n$ ($j=1$ is similar).
As above let $p \in (x_n,u_+)$ and $q = h(p) \in (y_n,v_+)$.
As before we only treat the case $f^\prime (x_n)>0$.
The main difference here is that orbits of $\dot{x} = f(x)$ reach the right endpoint of $U$ in finite time,
and similarly for $\dot{y} = g(y)$.
To accommodate this let $t_1, s_1 > 0$ be such that $\phi_{t_1}(p) = u_1$ and $\psi_{s_1}(q) = v_1$.
If $t_1 \le s_1$ let $x_{n+1} = u_1$ and $y_{n+1} = \psi_{t_1}(q) \le v_1$,
and if $t_1 > s_1$ let $x_{n+1} = \phi_{s_1}(p) \le u_1$ and $y_{n+1} = v_1$.
Then the construction \eqref{eq:defh}
produces a $C^k$ conjugacy for all $x \in (p,x_{n+1})$,
and as explained above this is readily extended to $(x_{n-1},x_{n+1})$.
\end{proof}

Note that the intervals $(x_{j-1},x_{j+1})$ are precisely the basins of attraction of $x_j$ if $f^\prime(x_j)<0$ or the basin of repulsion of $x_j$ if $f^\prime(x_j)>0$.
Hence another way to phrase Theorem~\ref{thm:belit} is that there are local $C^k$ conjugacies on the basins of attraction and repulsion of corresponding stationary points.

\section{Differentiable conjugacy and saddle-node bifurcations}
\label{sect:diffcon}

The saddle-node bifurcation is associated with differential equations with a one-dimensional centre 
manifold reduction on which the equation $\dot{x} = f(x,\mu)$ satisfies the conditions (\ref{eq:gen}). In this section 
we sketch the conditions that need to hold in order for $C^k$ conjugacies to exist,
with proofs for the cases with no periodic orbits locally. Then in section~\ref{sect:normalforms}
we apply these ideas to the normal form  (\ref{eq:nf}). 

By the implicit function theorem there exists a unique local branch of stationary points of the form
\begin{equation}\label{eq:IFTsp}
\mu =-\frac{f_{xx}}{2f_\mu}x^2+O(|x|^3).
\end{equation}
This has two solutions in $\mu>0$ and none in $\mu <0$ for the choice of $f_\mu >0$ 
and $f_{xx} < 0$ made in \eqref{eq:gen}.

Now consider two $C^k$ families of differential equations, $f$ parametrised by $\mu$, and $g$ 
parametrised by $\nu$, both of which satisfy (\ref{eq:gen}).
Let $\phi_t(x,\mu)$ and $\psi_t(y,\nu)$ denote their flows.
If $\mu$ and $\nu$ are positive then by Theorem~\ref{thm:belit} there are local conjugacies
if it is possible to
choose $\nu = N(\mu)$ such that the multipliers of the corresponding stationary points are equal.
This is not a trivial condition.
In the next section we use the implicit function theorem 
to prove that the coefficient $a$ of the normal form (\ref{eq:nf}) and the parameter $\nu$ can be chosen as functions 
of $\mu$ so that these conditions do hold and so Theorem~\ref{thm:belit} can be applied.

If instead $\mu=\nu =0$ then both systems satisfy Takens' theorem (Theorem~\ref{thm:takens}) provided $f$ and $g$ are $C^\infty$.
The following theorem accommodates the case in which $\mu$ and $\nu$ are negative.

\begin{theorem}Suppose that $f$ and $g$ are $C^k$ families of differential equations, $f$ parametrised by $\mu$, and $g$ 
parametrised by $\nu$, both of which satisfy (\ref{eq:gen}). If $\mu<0$ and $\nu <0$ then there exist neighbourhoods $U$ and $V$ of zero and a continuous function $N(\mu)$ with $N(0)=0$ such that $\dot x = f(x,\mu )$ on $U$ is $C^k$-conjugate to $\dot y =g(y, N(\mu ))$ on $V$.
\label{thm:negpars}\end{theorem}

\begin{proof}
If $\mu <0$ and $\nu <0$ then
the flows are decreasing in a neighbourhood of the origin. Given neighbourhoods $U=(u_0,u_1)$ and $V=(v_0,v_1)$ 
of zero there exists $T>0$ such that for all $t>T$ there exist parameters $\mu(t)$ and $\nu(t)$ such that 
\[
\phi_t(u_1,\mu(t)) = u_0, \quad \psi_t(v_1,\nu(t))=v_0 \,.
\]
For sufficiently large values of $T$ the functions $\mu(t)$ and $\nu(t)$ are increasing and limit to zero, hence we can invert 
one of them, $t=t(\mu)$, to obtain a correspondence $\nu = \nu(t(\mu)) = N(\mu)$ between $\mu$ and $\nu$.

At corresponding parameters the systems are $C^k$-conjugate on $U$ and $V$ by the $C^k$ flow box theorem \cite[Theorem 1.1]{PdM1982}.
Equivalently, as this is trivial in one dimension,
for all $x\in U$ there exists $\tau$ such that $x = \phi_\tau(u_1,\mu)$, and the $C^k$ conjugating function $y=h(x)$ can be defined by 
\[
h(x) = \psi_\tau(v_1,N(\mu)).
\]
 \end{proof}

\section{Normal forms and Takens' coefficient}\label{sect:normalforms}
In this section we will prove the claims in the introduction. 
The analysis closely follows
those in \cite{GS2022} for the discrete-time setting.
There are four parts to the calculation: an asymptotic computation of the position of the stationary points of the
general system (\ref{eq:gen}), the equivalent computation for the normal form (\ref{eq:nf}),
the use of the implicit function theorem to
identify functions $\nu(\mu)$ and $a(\mu)$ 
for the parameter and coefficient of the normal form at which the corresponding stationary points have
equal multipliers, and finally the use of Theorem~\ref{thm:belit} to establish the existence of 
local differentiable conjugacies. Throughout this section (\ref{eq:gen}) can be used to write the 
differential equation for $f$ (with all partial derivatives evaluated at $x = \mu = 0$ and assuming that $f$ is at least $C^3$) as
\begin{equation}\label{eq:fTaylor}
\dot x = f_\mu \mu +\frac{1}{2} \!\left( f_{xx}x^2+2f_{x\mu}x\mu + f_{\mu\mu}\mu^2 \right)+ \frac{1}{6}f_{xxx}x^3 + \cdots ,
\end{equation} 
and by assumption
\begin{equation}\label{eq:condsgen}
f_\mu >0\quad \textrm{and}\quad f_{xx}<0.
\end{equation}

\subsection{Stationary points of $f$}\label{subsect:spf}
From (\ref{eq:IFTsp}) we know there are two stationary points if $\mu >0$ and that their positions are 
functions of $m=\sqrt{\mu}$. Expanding in power series in $m$, substituting into (\ref{eq:fTaylor}), and
setting the right hand side equal to zero, gives the two solutions $x_{1,2}$ as (for $r=1,2$ below)
\begin{equation}\label{eq:fpf}
x_{r}=(-1)^r \sqrt{\frac{-2f_\mu}{f_{xx}}} \,m+\left(\frac{f_\mu f_{xxx}-3f_{x\mu}f_{xx}}{3f_{xx}^2}\right)m^2+O(m^3),
\end{equation}
assuming $f$ is at least $C^4$.
Differentiating (\ref{eq:fTaylor}) with respect to $x$ and evaluating at $x_r$ gives the multipliers
\begin{equation}\label{eq:fmult}
f'(x_r) = (-1)^{r+1}\sqrt{-2f_\mu f_{xx}} \,m-\frac{2}{3}\frac{f_\mu f_{xxx}}{f_{xx}}m^2+O(m^3).
\end{equation}

\subsection{Stationary points of the normal form}\label{subsect:spg}
The normal form (\ref{eq:nf}) is $\dot y=g(y)$ with
\begin{equation}\label{eq:nfg}
g(y)=\nu -y^2+ay^3.
\end{equation}
Hence 
\[
g_\nu =1, 
\quad g_{yy}=-2, \quad g_{yyy}=6a,
\]
and all other derivatives are zero. From (\ref{eq:fpf}) and (\ref{eq:fmult}) we can read off the values of 
the stationary points and their multipliers with $\nu =n^2 >0$: 
\begin{equation}\label{eq:fpg}
y_{r}= (-1)^r n +\frac{1}{2} a n^2 +O(n^3),
\end{equation} 
and
\begin{equation}\label{eq:gmult}
g^\prime (y_r)= 2 (-1)^{r+1} n+2an^2+O(n^3) .
\end{equation}

\subsection{Equality of the multipliers}\label{subsect:equality}
We require $f'(x_r) = g'(y_r)$ for $r=1$ and $r=2$.
These two equations are
enough to determine $\nu$ and $a$ as functions of $\mu$. However, their leading order 
terms are equal (up to sign) and so in order to use the implicit function theorem
it is necessary to solve two related functions.

Anticipating that $n$ is proportional to $m$, define $p$ by $n=mp$. Now (again a standard trick, see e.g. \cite{Devaney, GS2022}), let
$G_r(a,p,m)=f^\prime (x_r)-g^\prime (y_r)$ and define
\begin{equation}\label{eq:defF1}
F_1(a,p,m)=\begin{cases}\frac{G_1(a,p,m)}{m} & \textrm{if}~m\ne 0,\\ \frac{\partial G_1}{\partial m}(a,p,m) & \textrm{if}~m=0.
\end{cases}\end{equation}
From (\ref{eq:fmult}) and (\ref{eq:gmult}) with $n=pm$
\begin{equation}\label{eq:F1leading}
F_1(a,p,m)=\sqrt{-2f_\mu f_{xx}}-2p +O(m),
\end{equation}
and observe $F_1(a,p,0)=0$ if $p=p_0$ with
\begin{equation}\label{eq:p0}
p_0=\sqrt{-\frac{1}{2}f_\mu f_{xx}}.
\end{equation}
Now let
\begin{equation}\label{eq:defF2}
F_2(a,p,m)=\begin{cases}\frac{G_1(a,p,m)+G_2(a,p,m)}{m^2} & \textrm{if}~m\ne 0,\\ \frac{1}{2}\frac{\partial^2 ~}{\partial m^2}(G_1(a,p,m)+G_2(a,p,m)) & \textrm{if}~m=0.
\end{cases}\end{equation}
Once again, from (\ref{eq:fmult}) and (\ref{eq:gmult}) with $n=pm$
\begin{equation}\label{eq:F2leading}
F_2(a,p,m)=-\frac{4}{3}\frac{f_\mu f_{xxx}}{f_{xx}} -4ap^2 +O(m),
\end{equation}
and observe $F_1(a,p,0)=0$ if $p=p_0$ and $a=a_0$ with
\begin{equation}\label{eq:a0n}
a_0=\frac{1}{p_0^2}\frac{f_\mu f_{xxx}}{3f_{xx}}=\frac{2f_{xxx}}{3f_{xx}^2}.
\end{equation}
With these choices of $p_0$ and $a_0$ 
\[
F_1(a_0,p_0,0)=0, \quad F_2(a_0,p_0,0)=0.
\]
To apply the implicit function theorem the functions $F_1$ and $F_2$ must be at least $C^1$, hence $f$ must be at least $C^4$.
To obtain a unique local solution (valid in $m>0$) the determinant of the matrix $M$ of partial derivatives of $F_1$ and $F_2$ must not vanish.
From (\ref{eq:F1leading}) and (\ref{eq:F2leading})
\[
M = \left. \begin{bmatrix}
\frac{\partial F_1}{\partial a} & \frac{\partial F_1}{\partial p}\\
\frac{\partial F_2}{\partial a} & \frac{\partial F_2}{\partial p}
\end{bmatrix} \right|_{(a,p,m) = (a_0,p_0,0)}=\left(\begin{array}{cc}0 & -2\\ -4p_0^2 & -8a_0p_0\end{array}\right),
\]
so $\det(M) = -8 p_0^2 \ne 0$.
Thus indeed
there is a unique local curve of solutions with $n=p_0m+O(m^2)$, i.e. $\nu=p_0^2\mu + O \big( \mu^{\frac{3}{2}} \big)$,
and $a=a_0+O(m)$.  

\subsection{Local differentiable conjugacies}
We can now put these calculations together with the strategy of section~\ref{sect:diffcon} to determine the 
relation between the original map $f$ and the normal form $g$.

\begin{theorem}If $f$ is $C^k$, $k\ge 4$, then there exist neighbourhoods
$U$ and $V$ of zero and $\mu_0>0$ such that
\begin{enumerate}
\item If $\mu\in (-\mu_0,0)$ then there exists a continuous function $\nu (\mu )$ such that $\dot x=f(x)$ 
on $U$ at $\mu$ is $C^k$ conjugate to $\dot y=g(y)$ on $V$ at $\nu (\mu )$ for any value of $a$.
\item If $\mu =0$ and $f$ is $C^\infty$ then $\dot x=f(x)$ on $U$ is $C^\infty$ conjugate to $\dot y=g(y)$ on $V$ if $a=a_0$.
\item If $\mu\in (0,\mu_0)$ then there are continuous functions $\nu$ and $a$ with
\[\nu=p_0^2\mu + O(\sqrt{\mu}^3)\quad  \textrm{and}\quad  a=a_0+O(\sqrt{\mu}),\]
and $p_0$ and $a_0$ given by (\ref{eq:p0}) and (\ref{eq:a0n}), such that  $\dot x=f(x)$ is $C^k$ conjugate to $\dot y=g(y)$ on the basins of attraction and repulsion of the corresponding stationary points in $U$ and $V$.
\end{enumerate}
\label{thm:diffconjnf}\end{theorem}

This simply collects together the results established earlier. Case (i) is Theorem~\ref{thm:negpars},
Case (ii) is Takens' theorem (Theorem~\ref{thm:takens}),
and case (iii) follows from the results in this section together with Theorem~\ref{thm:belit}.

Since the value of $a$ is unimportant for the first part of this statement, we may choose $a(\mu )=a_0$ if $\mu <0$ so that
$a(\mu )$ is continuous at $\mu =0$ in the $C^\infty$ case. We conjecture that Theorem~\ref{thm:diffconjnf}(ii) remains true in the $C^k$ case, $k\ge 4$. This conjecture is certainly true in the discrete time case \cite{GS2022,KCG1990}. 

\subsection{A remark on signs}
Although $f_\mu (0,0) >0$ and $f_{xx}(0,0)<0$ may be assumed without loss of generality for a generic saddle-node bifurcation, it is useful to have the constants $p_0$ and $a_0$ in a form which can be applied without making the explicit transformation to this case. Clearly
\begin{equation}\label{eq:p0gen}
p_0^2=\frac{1}{2}|f_\mu f_{xx}|,
\end{equation}
regardless of the signs of $f_\mu$ and $f_{xx}$, but the sign of $a_0$ needs a little more thought. By considering the transformation $x\to -x$ we see that in the differential equation $f_{xx}\to -f_{xx}$ and $f_{xxx}\to f_{xxx}$. Hence
$a_0$ is invariant under the transformation and no adjustment is necessary.

\section{Fraedrich's temperature model}\label{sect:Fraedrichs}
Saddle-node bifurcations are a mechanism for tipping points in climate dynamics.
One such saddle-node bifurcation is exhibited by Fraedrich's model
for temperature, T(t); see \cite{Ashwin2012,Fraedrich} for details and justification.
The temperature variation has a black body radiation term of order $T^4$,
a constant solar warming term, and a $T^2$-term describing the variation of ice albedo with temperature.
The one dimensional model is\footnote{
Ashwin {\em et.~al.}~\cite{Ashwin2012} include an extra factor $\frac{1}{c}$ in the ODE which we believe should be absent,
likely a simple algebra mistake that arose when simplifying equation (4.1) of \cite{Fraedrich}.
}
\begin{equation}\label{eq:Fraedrichs}
\dot{T} = a \!\left( -T^4 + b \mu T^2 - d \mu \right),
\end{equation}  
and the parameters are 
\[
a = \frac{e_{SA}\sigma}{c}, \quad b=\frac{b_2I_0}{4e_{SA}\sigma}, \quad d=\frac{(a_2-1)I_0}{4e_{SA}\sigma}.
\]
Here the more fundamental constants are the Stefan-Boltzmann constant $\sigma$,
the insolation $I_0$, the oceanic thermal capacity $c$, and the effective emissivity $e_{SA}$.
The constants $a_2 > 1$ and $b_2$ model the ice albedo effect: the albedo is $a_2-b_2T^2$,
and $\mu$ is a parameter that describes variations due to changes in the planetary orbit or insolation values.
Reasonable values (taken from \cite{Ashwin2012}) are
\begin{equation}\label{eq:Fraepars}\begin{array}{c}
I_0=1366\, Wm^{-2}, \quad  \sigma= 5.6704 \times 10^{-8}\, Wm^{-2} K^{-4},  \quad c =108\, kgKs^{-2},\\
e_{SA}=0.62, \quad a_2= 1.6927, \quad b_2 = 1.690 \times  10^{-5}\, K^{-2}.
\end{array}\end{equation}
The parameter $\mu$ is order one and
acts as the bifurcation parameter.
  
Stationary points are solutions of the quadratic equation for $T^2$,
\[
T^4-b\mu T^2+d\mu =0,
\]
so the solutions are 
\[
T^2=\frac{1}{2}\!\left( b\mu\pm \sqrt{b^2\mu^2-4d\mu}\right).
\]
Two solutions are created in a saddle-node bifurcation as $\mu$ increases through $\mu_c= \frac{4d}{b^2}$ with temperature
$T_c=\sqrt{\frac{1}{2}b\mu_c} = \sqrt{\frac{2 d}{b}}$.
Now let $f(T,\mu)$ denote the right hand side of \eqref{eq:Fraedrichs}.
A routine calculation shows that at the bifurcation point
\[
f_\mu = a d, \quad
f_{TT} = -\frac{16 a d}{b}, \quad
f_{TTT} = -24 a T_c \,.
\]
This means that the two quantities, $p_0$ and $a_0$, that drive the normal form are related to the basic parameters by
\begin{equation}\label{eq:a0p0}
\begin{split}
p_0^2 &= -\frac{1}{2} f_\mu f_{TT}
= \frac{8 a^2 d^2}{b}
= \frac{2 e_{SA} \sigma I_0 (a_2 - 1)^2}{b_2 c^2}, \\
a_0 &= \frac{2 f_{TTT}}{3 f_{TT}^2}
= -\frac{1}{8 \sqrt{2} a} \left( \frac{b}{d} \right)^{\frac{3}{2}}
= -\frac{c}{8 \sqrt{2} e_{SA} \sigma} \left( \frac{b_2}{a_2 - 1} \right)^{\frac{3}{2}}.
\end{split}
\end{equation}
An interesting feature of this analysis us that Takens' coefficient, $a_0$ has two factors.
The factor $\frac{1}{a}$ represents a scaling in time since in the new time $\tau =at$ the multiplicative factor in (\ref{eq:Fraedrichs}) is scaled to unity. The second (and indeed the last ratio in the expression for $p_0^2$) shows a nonlinear dependence on the level and sensitivity of the albedo effect to changes in global temperature.

\section{Thermohaline circulation: perturbation theory}\label{sect:example2}
Stommel's two-box model of the oceanic circulation \cite{Stommel1961} remains a useful motivating example for the possibility of dramatic non-reversible changes in the climate due to global temperature rise. In the version used by Cessi \cite{Cessi1994} this can be written using a cubic nonlinearity instead of the modulus in the original model, so the non-dimensionalised temperature gradient $x$ and salinity gradient $y$ between the two boxes evolves according to the equations
\begin{equation}\label{eq:stomcessi}
\begin{split}
\dot x &= -\alpha (x-1)-x \!\left(1+m(x-y)^2\right),\\ 
\dot y & = p-y \! \left(1+m(x-y)^2\right).
\end{split}\end{equation}
Here $\alpha$ is a non-dimensionalised time constant, the ratio between the diffusive and temperature relaxation time scales,
and is approximately 3600, $m$ is the ratio of the diffusive time scale to the advective timescale and is approximately 7.5, and $p$ is the non-dimensional freshwater flux with an average of order one. These estimates of magnitude are taken from \cite{Cessi1994} where $p$ is allowed to oscillate stochastically about its mean and $m$ is denoted by $\mu^2$, which we have changed to avoid confusion with general bifurcation parameters. Here we take $p$ to be constant and treat it as the bifurcation parameter (cf. \cite{Stommel1961}). Kuehn \cite{Kuehn2013} describes a slow-fast manifold approach to the same problem.

Since $\alpha$ is much larger than $m$ and $p$, a perturbation theory approach can be taken with solutions expanded as power series in $\alpha^{-1}$ giving
$x=1+O(\alpha^{-1})$ and
\begin{equation}\label{eq:firstorder}
\dot y = p-y\!\left(1+m(1-y)^2\right) +O(\alpha^{-1}).
\end{equation}
Writing $y=y_0+O(\alpha^{-1})$, the perturbation equation for $y_0$ is obtained by ignoring the order $\alpha^{-1}$ terms, so, dropping the subscript $0$ from here on, saddle-node bifurcations of (\ref{eq:firstorder}) at leading order occur if
\begin{equation}\label{eq:consn}
p=y\!\left(1+m(1-y)^2\right) \quad \text{and} \quad 1+m-4my+3my^2=0.
\end{equation}
The second equation of (\ref{eq:consn}), which comes from setting the derivative of the right hand side of (\ref{eq:firstorder}) equal to zero,
is a quadratic for $y=y(m)$.
By solving it for $y$ we obtain
\[
y_\pm (m)=\frac{1}{3}\left(2\pm\sqrt{1-\frac{3}{m}}\right), \quad m>3,
\]
and by substituting this into the first equation of \eqref{eq:consn} we obtain (after simplification)
\begin{equation}\label{eq:locus}
p_\pm =\frac{2}{3}+\frac{2m}{27}\left[ 1 \mp \left(1-\frac{3}{m}\right)^{\frac{3}{2}}\right].
\end{equation}
Note that $p_+<p_-$. If $p\in (p_+,p_-)$ then there are three solutions, see Fig.~\ref{fig:boxModelBifDiag}. The two stable solutions correspond to those with highest and lowest salinity gradient. The ocean circulation currently lies on the solution with higher salinity gradient, and if $p>p_-$ this is the only stable stationary point, lending weight to arguments that higher freshwater flux stabilises the ocean circulation. However, if $p$ decreases through $p_+$ then this solution disappears and the solution would move rapidly to the stationary point with lower gradient and weaker circulation. This would have rapid and severe consequences for weather on the eastern European Atlantic coast. 

\begin{figure}[h!]
\begin{center}
\includegraphics[height=4cm]{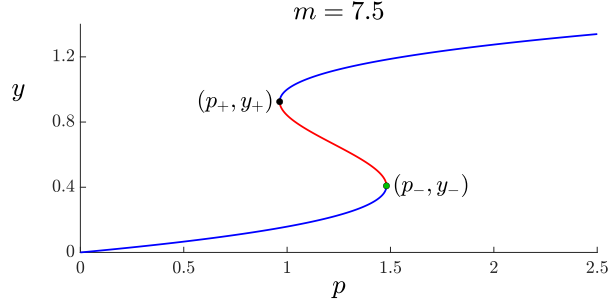}
\caption{
A bifurcation diagram of the box model \eqref{eq:stomcessi}
in the limit $\alpha \to \infty$ and with $m = 7.5$.
Branches of stable [unstable] equilibria are coloured blue [red].
Two saddle-node bifurcations are indicated.
\label{fig:boxModelBifDiag}
} 
\end{center}
\end{figure}

We therefore concentrate on the tipping point $p_+$. If $f$ represents the leading order term in the differential equation (\ref{eq:firstorder}) then
\[
f_p=1, \quad f_{yy}=4m-6my, \quad f_{yyy}=-6m.
\]
Observe $f_p > 0$ and $f_{yy} < 0$ at $y_+$, thus the bifurcation occurs as $p$ increases.
The determining bifurcation numbers
(evaluating at $y_+$) are
\begin{equation}\label{eq:stommelpars}
\begin{split}
p_0^2 &=
-\frac{1}{2} f_p f_{yy}
=\sqrt{m(m-3)}, \\
a_0 &= 
\frac{2 f_{yyy}}{3 f_{yy}^2}
=\frac{-1}{m-3}.
\end{split}
\end{equation}
Thus Takens' coefficient
is large if $m$ is only slightly
greater than 3, and if $m=7.5$ it is approximately $0.22$
which is fairly small so
shows that at the currently assumed parameter values the bifurcation is relatively symmetric.

\begin{figure}[h!]
\begin{center}
\includegraphics[height=9cm]{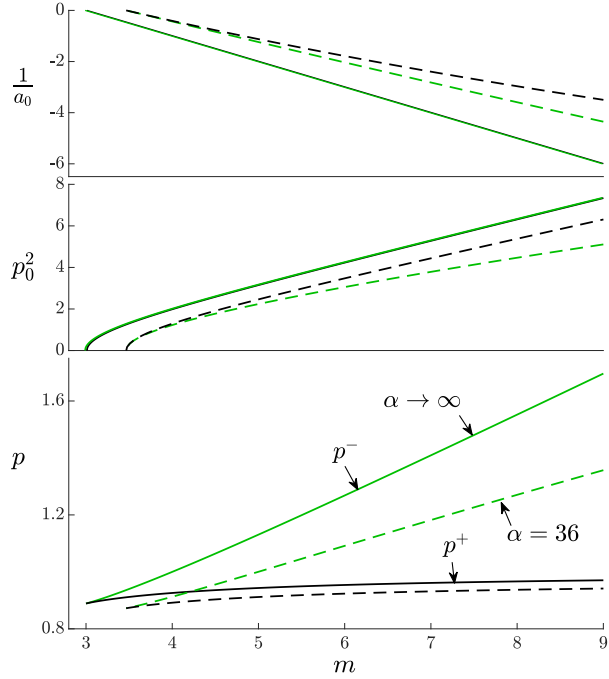}
\caption{
The bottom plot is a two-parameter bifurcation diagram showing branches of saddle-node bifurcations
for the box model \eqref{eq:stomcessi}.
Solid curves correspond to the limit $\alpha \to \infty$ (given explicitly by \eqref{eq:locus});
dashed curves are for $\alpha = 36$ (computed by numerical continuation).
The other plots show $p_0^2 = \frac{1}{2} |f_p f_{yy}|$ and $\frac{1}{a_0}$,
where $a_0 = \frac{2 f_{yyy}}{3 f_{yy}^2}$.
\label{fig:boxModelBifSet}
} 
\end{center}
\end{figure}

The bifurcation numbers are plotted in Fig.~\ref{fig:boxModelBifSet} for both $y_-$ and $y_+$.
The solid curves in the bottom plot are the two branches $p^\pm$ \eqref{eq:locus} of saddle-node bifurcations in the $(m,p)$ parameter plane.
These meet in a cusp bifurcation at $(m,p) = \left( 3, \frac{8}{9} \right)$.
The other plots show $p_0^2$ and $\frac{1}{a_0}$ which
are identical for the two branches.
Notice both $p_0^2$ and $\frac{1}{a_0}$ converge to $0$ at the cusp bifurcation.

\section{Centre Manifolds}\label{sect:cm}
In higher dimensional problems the one-dimensional system describing the saddle-node bifurcation is the projection of the flow onto a centre manifold.
In this section we use a centre manifold reduction to how the bifurcation numbers $p_0$ and $a_0$ can be computed for a two-dimensional system.

Consider a $C^k$ ($k \ge 4$) system
\begin{equation}\label{eq:orig}
\dot{x} = F(x,y,\mu), \quad \dot{y} = G(x,y,\mu),
\end{equation}
where $F(0,0,0) = G(0,0,0) = 0$
and the eigenvalues of the Jacobian matrix of \eqref{eq:orig} evaluated at $(x,y,\mu) = (0,0,0)$ are $0$ and $\lambda \ne 0$.
If $\lambda < 0$ the local centre manifold is stable.
Assume a coordinate transformation has been done so that the Jacobian matrix is in Jordan form, specifically
\begin{equation}
\begin{bmatrix}
\frac{\partial F}{\partial x} & \frac{\partial F}{\partial y} \\
\frac{\partial G}{\partial x} & \frac{\partial G}{\partial y}
\end{bmatrix} =
\begin{bmatrix} 0 & 0 \\ 0 & \lambda \end{bmatrix}.
\nonumber
\end{equation}
Then the system can be written as
\begin{equation}
\begin{split}
\dot{x} &= b_0 \mu + b_1 x^2 + b_2 x y + b_3 y^2 + b_4 \mu x + b_5 \mu y + b_6 \mu^2 + b_7 x^3 + \cdots, \\
\dot{y} &= \lambda y + c_1 x^2 + c_2 x y + c^3 y^2 + c_4 \mu x + c_5 \mu y + c_6 \mu^2 + c_7 x^3 + \cdots,
\end{split}
\label{eq:2dfull}
\end{equation}
where the neglected terms are cubic or higher order in $x$, $y$, and $\mu$,
except we have explicitly written the $x^3$ terms as one of them will be needed below.
There is an extended centre manifold $y = h(x,\mu)$ obtained formally by adding the trivial equation $\dot{\mu} = 0$.
This manifold has the form
\begin{equation}
y = d_1 x^2 + d_2 \mu x + d_3 \mu^2 + O \!\left( (|x|+|\mu|)^3 \right),
\nonumber
\end{equation}
and invariance implies
\begin{equation}
\dot{y} = (2 d_1 x + d_2 \mu) \dot{x} + O \!\left( (|x|+|\mu|)^2 \right).
\nonumber
\end{equation}
By substituting \eqref{eq:2dfull} and equating the coefficients of $x^2$ we find
$d_1 = -\frac{c_1}{\lambda}$
(formulas for $d_2$ and $d_3$ will not be needed).
Then by substituting our expression for the extended centre manifold into \eqref{eq:2dfull}, we obtain
the leading order equation on the centre manifold:
\begin{equation}
\dot{x} = b_0 \mu + b_1 x^2 + b_4 \mu x + b_6 \mu^2 + \left( b_7 - \frac{b_2 c_1}{\lambda} \right) x^3 + \cdots.
\label{eq:leadcm}
\end{equation}
Finally by using the formulas given earlier, where now $f$ denotes the right hand side of \eqref{eq:leadcm},
\begin{equation}
\begin{split}
p_0^2 &= \frac{1}{2} |f_\mu f_{xx}|
= |b_0 b_1|
= \frac{1}{2} |F_\mu F_{xx}|, \\
a_0 &= \frac{2 f_{xxx}}{3 f_{xx}^2}
= \frac{1}{b_1^2} \left( b_7 - \frac{b_2 c_1}{\lambda} \right)
= \frac{1}{3 F_{xx}^2} \left( 2 F_{xxx} + \frac{6 F_{xy} G_{xx}}{\lambda} \right).
\end{split}
\label{eq:p0a0cm}
\end{equation}
Notice \eqref{eq:leadcm} has an additional term in the cubic coefficient
and hence Takens' coefficient $a_0$ is modified by the centre manifold reduction even though the
coefficients $b_0$ and $b_1$ are unaltered.
In other words, Takens' coefficient contains information about the curvature of the centre manifold
and the strength of the linear contraction or expansion in coordinates orthogonal to the centre subspace.

\section{Thermohaline circulation: centre manifold analysis}\label{sect:stommelcm}
Returning to the model of thermohaline convection (\ref{eq:stomcessi}) of section~\ref{sect:example2} the centre manifold
analysis of section~\ref{sect:cm} makes it possible to extend the more refined description of the saddle-node bifurcations away from the large $\alpha$ limit.
This involves three steps. First, the bifurcation values of the parameters and variables need to be determined.
Then a centre manifold reduction is calculated, and finally the coefficients of the equation on the centre manifold need to be evaluated.
Section~\ref{sect:cm} shows that for two-dimensional systems the last step is simple once the transformation to Jordan normal form is made at the bifurcation point:
$p_0$ and $a_0$ can be evaluated directly using (\ref{eq:p0a0cm}).
However, as in many problems the first step is either algebraically messy or impossible.
Therefore in this section we will
also use numerical methods to evaluate $p_0$ and $a_0$.

We start by numerically continuing the branches of saddle-node bifurcations.
With $\alpha = 3600$, the value of $\alpha$ used by Cessi \cite{Cessi1994},
we observed these branches to be indistinguishable,
on the scale of Fig.~\ref{fig:boxModelBifSet},
from the curves (\ref{eq:consn}) obtained earlier.
So for illustration we have instead used $\alpha = 36$
which results in the dashed curves shown in Fig.~\ref{fig:boxModelBifSet}, bottom plot.

At points on these branches the derivatives of the right hand sides of the differential equation can be evaluated.
Then the eigenvectors of the linear part can be used to bring the system into Jordan form
and hence to compute
derivatives in the coordinates of \eqref{eq:2dfull}.
Finally \eqref{eq:p0a0cm} is evaluated numerically.

The values of
$p_0^2$ and $\frac{1}{a_0}$ obtained using this process are shown in Fig.~\ref{fig:boxModelBifSet}, middle and top plots.
Overall the results with the smaller value $\alpha = 36$ are still reasonably close to the $\alpha \to \infty$ limit;
notice the values of $p_0$ and $a_0$ now differ for the two branches.

\section{Conclusion}\label{sect:conclusion}
This paper describes a deeper connection between truncated normal forms and bifurcations. As in the 
discrete time case\cite{GS2022}, the local conjugacy can be made differentiable on the basins of 
attraction and repulsion of the stationary solutions provided the cubic coefficient of the normal form
is chosen judiciously. This means that (for example) rates of convergence are preserved.  The
cubic term is described by Takens' coefficient, and this contains extra information about asymmetry
between the two branches and proximity to cusp bifurcations. We believe that Takens' coefficient should
be calculated as a matter of course in many analyses of saddle-node bifurcations for this extra 
information. The other coefficient identified, $p_0$, describes the speed of the bifurcation. The extra
information in $p_0$ and $a_0$ allow different
saddle-node bifurcations to be compared in a mathematically meaningful way. Since the parameter dependence of 
these coefficients can be calculated explicitly using the functional relationships of section~\ref{sect:normalforms}, this variation, and particularly the speed of variation, can also provide information about the evolution of solutions away from the bifurcation point. 


The other bifurcations on one-dimensional centre manifolds, the transcritical and pitchfork bifurcations,
can be approached in the same way. The details are similar to the discrete time case \cite{GS2022}; we have concentrated on the generic case here.
Note that for the pitchfork bifurcation two extra terms are needed in the truncated normal form (in addition to the standard cubic term) so that stability coefficients can be matched at three stationary points.

The extended truncated normal forms provide significantly more information than the standard simple truncations, 
and have the added advantage that the techniques to identify the relationship between the parameterisations of the
original system and the model are purely algebraic. They rely on solving two or three additional equations, 
solutions of which can be guaranteed by the implicit function theorem. Since the implicit function theorem is 
already a standard tool in bifurcation theory, little extra information beyond Sternberg's linearisation theorem
has been needed to introduce the extended analysis.

\section*{Acknowledgements}
\setcounter{equation}{0}

The authors were supported by Marsden Fund contract MAU1809,
managed by Royal Society Te Ap\={a}rangi.


\end{document}